\documentclass[11pt]{article}
\usepackage{amstext, amsmath,latexsym,amsbsy,amssymb,amsmath}

\usepackage[utf8]{inputenc} 
\usepackage{amssymb, amscd, mathrsfs,bbm}
\usepackage{graphicx} 

 \usepackage[parfill]{parskip} 
 
\usepackage{booktabs} 
\usepackage{array} 
\usepackage{paralist} 
\usepackage{verbatim} 
\usepackage{subfig} 
\usepackage{mathrsfs}
\usepackage{amssymb}
\usepackage{amsthm}
\usepackage{amsmath,amsfonts,amssymb,esint}
\usepackage{graphics,color}
\usepackage{enumerate}
\usepackage{mathtools}

\numberwithin{equation}{section}

\newtheorem{theorem}{Theorem}[section]
\newtheorem{lemma}[theorem]{Lemma}
\newtheorem{proposition}[theorem]{Proposition}

\newtheorem{remark}[theorem]{Remark}

\newtheorem{assumption}[theorem]{Assumption}

\usepackage{color}

\newcommand{\beqn}{\begin{equation}}
\newcommand{\eeqn}{\end{equation}}
\newcommand{\bear}{\begin{eqnarray}}
\newcommand{\eear}{\end{eqnarray}}
\newcommand{\bean}{\begin{eqnarray*}}
\newcommand{\eean}{\end{eqnarray*}}

\usepackage{xcolor}

\def\varep{\varepsilon}

\def\f12{{\frac 1 2}}

\def\f{\widetilde{f}}

\newcommand{\bea}{\begin{eqnarray}}
\newcommand{\eea}{\end{eqnarray}}
\def\beaa{\begin{eqnarray*}}
\def\eeaa{\end{eqnarray*}}

\def\ba{\begin{array}}
\def\ea{\end{array}}

\def\r_h{{r_{\mathcal{H}}}}
\allowdisplaybreaks

\begin{document}
\title{Optimal local well-posedness theory for the kinetic wave equation}

\author{Pierre Germain \footnotemark[1] \and Alexandru D. Ionescu\footnotemark[2]   \and Minh-Binh Tran\footnotemark[3] 
}

\renewcommand{\thefootnote}{\fnsymbol{footnote}}
\footnotetext[1]{Courant Institute of Mathematical Sciences, 251 Mercer Street, New York, NY 10012, USA. \\Email: pgermain@cims.nyu.edu
}

\footnotetext[2]{Department of Mathematics, Fine Hall, Washington Road, Princeton, NJ 08544, USA.\\Email:  aionescu@math.princeton.edu
}

\footnotetext[3]{Department of Mathematics, University of Wisconsin-Madison, Madison, WI 53706, USA. \\Email: mtran23@wisc.edu
}

\maketitle
\begin{abstract} We prove local existence and uniqueness results for the (space-homogeneous) 4-wave kinetic equation in wave turbulence theory. We consider collision operators defined by radial, but general dispersion relations satisfying suitable bounds, and we prove two local well-posedness theorems in nearly critical weighted spaces. 
 \end{abstract}

{\bf Keyword:}  wave (weak) turbulence, quantum Boltzmann, nonlinear Schr\"odinger, wave kinetic equations

{\bf MSC:} {35BXX ; 37K05 ; 35Q55 ; 35C07; 45G05 ; 35Q20 ; 35B40 ; 35D30}

\setcounter{tocdepth}{1}

\tableofcontents
\section{Introduction}\label{Sec:Intro}

\subsection{Weak turbulence}

Weak turbulence refers to the theory descirbing nonequilibrium statistical mechanics of weakly nonlinear Hamiltonian systems; it is a universal phenomenon arising in a number of physical systems. 
For these systems, it is expected that the nonlinear effects lead to the stochastization of waves phases and a slow modulation of the amplitudes, and that a kinetic equation of quantum Boltzmann type for the mean square amplitudes can be written. There are two common types of such kinetic equations: the 3-wave and the 4-wave ones. The first derivation of a kinetic model of weak turbulence, which is a 3-wave one, was obtained, to our knowledge, in \cite{Peierls:1993:BRK,Peierls:1960:QTS} in the study of phonon interactions in anharmonic crystal lattices. We refer to  \cite{zakharov2012kolmogorov,Nazarenko:2011:WT,zakharov1998nonlinear,fitzmaurice1993nonlinear,newell2011wave,newell2013wave} for detailed discussions on the topics.

4-wave kinetic equations play an important role in the theory of weak turbulence and appear in several contexts: gravity and capillary waves on the surface of a finite-depth fluid \cite{zakharov1999statistical,Hasselmann:1962:OTN1,Hasselmann:1963:OTN2,hasselmann1966feynman,dyachenko1994free}, Alfven wave turbulence in astrophysical plasmas \cite{ng1996interaction}, optical waves of diffraction in nonlinear media \cite{DyachenkoNewellPushkarevZakharov:1992:OTW,lvov1998quantum,lvov2000finite}, quantum fluids \cite{kolmakov2014wave}, Langmuir waves  \cite{zakharov1972collapse} to name only a few.  

\subsection{The kinetic wave equation and its first properties}

The present article investigates the local well-posedness theory for the space-homogeneous 4-wave kinetic equation
\begin{equation}
\label{4wave}
\begin{aligned}
\partial_t f(t,p) \ =  & \ \mathcal{Q}[f](t,p), \mbox{ on } \mathbb{R}_+\times\mathbb{R}^3,\\
f(0,p) \  =  &f_0(p) \mbox{ on } \mathbb{R}^3.
\end{aligned}
\end{equation}
The trilinear operator $\mathcal{Q}$ is given by
$$
\mathcal{Q}[f](p) =   \iiint_{\mathbb{R}^{3\times3}}\delta(p+p_1-p_2-p_3)\delta(\omega + \omega_1 -\omega_2 - \omega_3) [f_2 f_3 (f_1+f)
-f f_1 (f_2 + f_3)]\,dp_1\,dp_2\,dp_3,
$$
where we denoted
$$
\omega= \omega(p), \qquad \omega_i = \omega(p_i), \qquad f = f(p), \qquad f_i = f(p_i).
$$
In the above, $p \mapsto \omega(p)$ is the dispersive relation of the underlying dispersive problem, to which we will come back shortly.

Notice that the nonlinear term can also be written
\begin{equation*}
\begin{aligned}
\mathcal{Q}[f](p)\ =  & \  \iiint_{\mathbb{R}^{3\times3}}\delta(p+p_1-p_2-p_3)\delta(\omega+\omega_1-\omega_2-\omega_3) 
\\ & \qquad\qquad\qquad\qquad\qquad\qquad \times f f_1 f_2 f_3 \Big[ \frac{1}{f} + \frac{1}{f_1} - \frac{1}{f_2} - \frac{1}{f_3} \Big] \,dp_1\,dp_2\,dp_3.
\end{aligned}
\end{equation*}
Writing the nonlinear term in this way makes it clear that the mass, momentum, and energy are formally conserved; they are defined respectively as
$$
\int_{\mathbb{R}^3} f(p)\,dp, \qquad \int_{\mathbb{R}^3}  p f(p)\,dp, \qquad \int_{\mathbb{R}^3}  \omega(p) f(p)\,dp.
$$
Furthermore, the entropy, defined by
\begin{equation*}
\int_{\mathbb{R}^3}  \log f(p) \,dp,
\end{equation*}
is formally decreasing. Finally, the above form of the nonlinear term leads to the stationary solutions
\begin{equation}
\label{stationary}
\frac{1}{\mu + \nu \cdot p + \xi \omega(p)},
\end{equation}
where $(\mu,\nu,\xi) \in \mathbb{R} \times \mathbb{R}^3 \times \mathbb{R}$ are such that $\mu + \nu \cdot p + \xi \omega(p) > 0$ for any $p$.

The equation~\eqref{4wave} does not admit invariant scalings for general dispersion relations $\omega(p)$. However, for $\omega(p) = |p|^2$, a number of scalings arises, which leave the set of solutions invariant. The most relevant one leaves the time variable untouched: it is given by the transformation
\begin{equation}
\label{scaling}
f(t,p) \mapsto \lambda^2 f(t, \lambda p).
\end{equation}

\subsection{The dispersion relation}

One of our aims is to allow more general dispersion relations which enjoy similar bounds to $\omega(p) = |p|^2$. This is motivated by the following instances of physical interest:
\begin{itemize}
\item The basic example is the Schr\"odinger case
\begin{equation}\label{Schr}
\omega(p) = |p|^2.
\end{equation}
\item The  Bogoliubov dispersion law \cite{E,newell2013wave} 
\begin{equation}\label{Bogoliubov0}
\vartheta(p)=\sqrt{\theta_1|p|^2+\theta_2|p|^4},
\end{equation}
where $\theta_1,\theta_2$ are strictly positive  constants.  
\item The modified Bogoliubov dispersion law \cite{E} and the Bohm-Pines dispersion law \cite{bohm1951collective}
\begin{equation}\label{Bogoliubov}
\vartheta(p)=\sqrt{\theta_0+\theta_1|p|^2+\theta_2|p|^4},
\end{equation}
where $\theta_0,\theta_1,\theta_2$ are strictly positive  constants. In the very low temperature regime \cite{EPV,IG,bohm1951collective}, $\vartheta$ can be replaced by the following approximated dispersion relation
\begin{equation}\label{BogoliubovLow}
\omega(p)={\lambda_0+\lambda_1|p|^2+\lambda_2|p|^4},
\end{equation}
 with  $\lambda_0,\lambda_1,\lambda_2$ being strictly positive constants depending on $\theta_0,\theta_1,\theta_2$. 
\end{itemize}

These examples are captured by the following general assumption.

\begin{assumption}\label{Assump1}

The dispersion relation is of the form
\begin{equation}\label{Dispersion}
\omega(p) \  = \ \Omega(|p|),
\end{equation}
and satisfies:
\begin{equation}\label{Omega}
\begin{aligned}
& \mbox{(i)          $\Omega(0) = 0$ (this is simply a convenient normalization).} \\
& \mbox{(ii)          $\Omega\in C^{1}(\mathbb{R}_+)$ and $\Omega(x)\ge 0$ for all $x$ in $\mathbb{R}_+$. }\\
& \mbox{(iii)          There exists a constant $c_1>0$ such that $\Omega'(x)\ge c_1x$, for all $x$ in $\mathbb{R}_+$. } \\
& \mbox{(iv)           There exists a constant $c_2>0$ such that $\Omega(x) \le \frac{1}{2} \Omega(c_2x)$, for all $x$ in $\mathbb{R}_+$.}
\end{aligned}
\end{equation}
\end{assumption}

\subsection{Rigorous results on the isotropic 4-wave kinetic equation and related models}

The first question is that of the derivation of this kinetic equation from Hamiltonian dynamics: it should arise in the weakly nonlinear, big box limit  under the random phase approximation. This is not the subject of this paper, but we refer to the classical textbooks~\cite{zakharov1998nonlinear, Nazarenko:2011:WT} for a heuristic discussion, as well as to~\cite{LukkarinenSpohn:2011:WNS} for the latest rigorous results.

The question of the local existence and uniqueness of solutions to \eqref{4wave} was first studied in \cite{EscobedoVelazquez:2015:OTT}, where the dispersion relation is of classical type $\omega(p)=|p|^2$, and the solution $f$ is radial (velocity-isotropic). Abusing notations by denoting $p$ for $|p|$ and $f(p)$ for $f(|p|)$, the equation \eqref{4wave} reduces to a one-dimensional Boltzmann equation
\begin{equation}\label{1D4wave}
\partial_t f=\int_{\mathbb{R}_+^2} \frac{p_2 p_3 \min\{p,p_1,p_2,p_3\}}{p} [f_2 f_3(f+f_1)- f f_1(f_2+f_3)]  \,dp_3 \, dp_4,
\end{equation}
where $p_1^2 = p_2^2+ p_3^2 - p^2$.

It is proved in~\cite{EscobedoVelazquez:2015:OTT} that the above equation admits global, measure valued, weak solutions. This functional framework allows in particular for condensation, namely the development of a point mass at the origin. It is furthermore showed that condensation can occur, and that, as $t \to \infty$, most of the energy is transfered to high frequencies. The articles \cite{kierkels2015transfer, kierkels2015self} are dedicated to a quadratic equation arising from \eqref{4wave} in the regime where a Dirac mass has formed, and contains most of the mass. Note that the existence and uniqueness of radial weak  solutions to  a slightly simplified version of the 4-wave kinetic equation for general power-law dispersion has been proved in \cite{merino2016isotropic}. 

The reduction to the radial model \eqref{1D4wave} is restricted to the case $\omega(p) = |p|^2$. It is therefore one of the goal of this paper to construct a local existence and uniqueness theory which would be valid in greater generality.

In the  theory of the classical Boltzmann equation, the conservation laws
\begin{equation}\label{ClassicalBoltzmann} p+p_1=p_2+p_3,  \ \ \ \ |p|^2+|p_1|^2=|p_2|^2+|p_3|^2 \end{equation}
play a very important role. Since \eqref{ClassicalBoltzmann} implies that $p$, $p_1,$ $p_2,$ $p_3$ are on the sphere centered at $\frac{p+p_1}{2}$ with radius $\frac{|p-p_1|}{2}$, the Boltzmann collision operators can be considered as integrals on spheres (see, for instance \cite{MR1942465,Cercignani:1988:BEI}) and the Carleman representation  \cite{Carleman:1933:TEI} can be used. This is not the case for more general dispersion relations, for which the resonant manifolds do not admit such simple parameterizations.

Let us mention that  \eqref{4wave} is very similar to the Boltzmann-Nordheim (Uehling-Ulenbeck) equation (cf. \cite{Nordheim:OTK:1928,UehlingUhlenbeck:TPI:1933}), which describes the evolution of the density function of a dilute Bose gas at high temperature (above the Bose-Einstein condensate transition temperature)
\begin{equation}
\label{QB}
\begin{aligned}
\partial_t f(t,p) \ =  & \ \mathcal{Q}[f](t,p) + \mathcal{Q}_0[f](t,p), \\
\mathcal{Q}_0[f](t,p)\ =  & \  \iiint_{\mathbb{R}^{3\times3}}\delta(p+p_1-p_2-p_3)\delta(\omega+\omega_1-\omega_2-\omega_3)[f_2f_3-  
ff_1]dp_1dp_2dp_3, \\
f(0,p) \  =  &f_0(p).
\end{aligned}
\end{equation}
Notice that $\mathcal{Q}_0$ is the classical Boltzmann collision operator. The study of \eqref{QB} is also a subject of rapidly growing interest in the kinetic community (cf. \cite{arkeryd2016cauchy,EscobedoVelazquez:2015:OTT,EscobedoVelazquez:2015:FTB,SofferBinh1,SofferBinh2,ToanBinh,MR1837939,Lu:2004:OID,Lu:2005:TBE,Lu:2013:TBE,BriantEinav:2016:OTC,JinBinh,saint2004kinetic,li2017global,reichl2017kinetic} and the references therein).  Thanks to the stabilization effect of the classical Boltzmann collision operator $\mathcal{Q}_0$, the classical method of moment production developed for the classical  Boltzmann equation can be applied (cf. \cite{BriantEinav:2016:OTC,li2017global}) to study the well-posedness of the equation \eqref{QB}. However, this method cannot be used for the 4-wave kinetic equation  since $\mathcal{Q}_0$ is missing there.

Besides the 4-wave kinetic equation, the 3-wave kinetic equation also plays an important role in the theory of weak turbulence, and has been studied in \cite{EscobedoBinh,AlonsoGambaBinh,GambaSmithBinh,CraciunBinh,EPV} for the phonon interactions in anharmonic crystal lattices, in \cite{GambaSmithBinh} for stratified  flows in the ocean, and in \cite{nguyen2017quantum} for capillary waves.

Finally, let us mention the (CR) equation,  which is derived in~\cite{faou2016weakly,buckmaster2016effective} and studied in~\cite{germain2015continuous,buckmaster2016analysis,germain2016continuous}, which is a Hamiltonian equation whoses nonlinearity is given by the trilinear term $\mathcal{T}_1$ (defined below).

\section{Main results}

For the sake of simplicity, we impose the abbreviation $f=f(t,p)$, $f_1=f_1(t,p)$, $f_2=f_2(t,p)$, $f_3=f_3(t,p)$ and $\omega=\omega(p)$, $\omega_1=\omega(p_1)$, $\omega_2=\omega(p_2)$, $\omega_3=\omega(p_3)$.  

We consider the initial-value problems in $\mathbb{R}^3\times [0,T]$ of the 4-wave kinetic equation
\begin{equation}\label{4waveEquivalence}
\begin{split}
\partial_t f  &=  \mathcal{Q}[f]:= \mathcal{T}_1(f,f,f)+ \mathcal{T}_2(f,f,f)-2\mathcal{T}_3(f,f,f),\\
f(0)&=f_0,
\end{split}
\end{equation}
where
\begin{equation}\label{Operators}
\begin{aligned}
\mathcal{T}_1(f,g,h) \ : = &\  \int_{\mathbb{R}^9}\delta(p+p_1-p_2-p_3)\delta(\omega+\omega_1-\omega_2-\omega_3)\times \\
& \times f(p_1)g(p_2)h(p_3)\,dp_1dp_2dp_3,\\
\mathcal{T}_2(f,g,h) \ : = &\  \int_{\mathbb{R}^9}\delta(p+p_1-p_2-p_3)\delta(\omega+\omega_1-\omega_2-\omega_3)\times \\
& \times f(p)g(p_2)h(p_3)\,dp_1dp_2dp_3,\\
\mathcal{T}_3(f,g,h) \ : = &\  \int_{\mathbb{R}^9}\delta(p+p_1-p_2-p_3)\delta(\omega+\omega_1-\omega_2-\omega_3)\times \\
& \times f(p)g(p_1)h(p_2)\,dp_1dp_2dp_3.
\end{aligned}
\end{equation}

We define the function spaces $L^r_s$, $r\in[1,\infty]$, $s\geq 0$ by the norms
\begin{equation}\label{nor1} 
\|f\|_{L^r_s} := \ \| \langle x \rangle^s f\|_{L^r},\qquad \langle x\rangle:=(1+|x|^2)^{1/2}.
\end{equation}
In the case $r=\infty$ we require also that $f$ is continuous, so we define
\begin{equation*}
L^\infty_s:=\{f\in C^0(\mathbb{R}^3):\,\|f\|_{L^\infty_s}<\infty\}.
\end{equation*}
Our first main theorem concerns local well-posedness of the initial-value problem \eqref{4waveEquivalence} in $L^\infty_s$, $s>2$. More precisely:

\begin{theorem}\label{theorem:main} (i) Assume that $\omega$ satisfies Assumption \ref{Assump1} and $s>2$. Then the initial-value problem \eqref{4waveEquivalence} is locally well-posed in $L^\infty_s$ for $s>2$, in the sense that for any $R>0$ there is $T\gtrsim_s R^{-2}$ such that for any initial-data $f_0\in L^\infty_s$ with $\|f_0\|_{L^\infty_s}\leq R$, there is a unique solution $f$ in $C^1([0,T]:L^\infty_s)$ of the initial-value problem \eqref{4waveEquivalence}. Furthermore, $\|f(t)\|_{L^\infty_s}\leq 2R$ for any $t\in[0,T]$ and the map $f_0 \mapsto f$ is continuous from $L^\infty_s$ to $C^1([0,T]:L^\infty_s)$.

(ii) If furthermore $f_0\geq 0$, then $f(t)$ is non-negative for any $t\in[0,T]$.
\end{theorem}

In the special Schr\"{o}dinger case, we prove also a stronger local-wellposedness theorem in  $L^2_s$, $s>1/2$. More precisely:

\begin{theorem}\label{theorem:main2} (i) Assume that $\omega(p)=|p|^2$ and $s>1/2$. Then the initial-value problem \eqref{4waveEquivalence} is locally well-posed in $L^2_s$ for $s>1/2$: for any $R>0$ there is $T\gtrsim_sR^{-2}$ such that for any initial-data $f_0\in L^2_s$ with $\|f_0\|_{L^2_s}\leq R$, there is a unique solution $f$ in $C^1([0,T]:L^2_s)$ of the initial-value problem \eqref{4waveEquivalence}. Furthermore, $\|f(t)\|_{L^2_s}\leq 2R$ for any $t\in[0,T]$ and the map $f_0 \mapsto f$ is continuous from $L^2_s$ to $C^1([0,T]:L^2_s)$.

(ii) If $f_0\geq 0$ then $f(t)$ is non-negative for any $t\in[0,T]$.
\end{theorem}

Theorems \ref{theorem:main} and \ref{theorem:main2} follow by fixed point arguments from the following propositions: 

\begin{proposition}
\label{PropLinfty}
Assume that $\omega$ satisfies Assumption~\ref{Assump1}, $s>2$, and $0 \leq \gamma < \min(s-2,1)$. Then the operators $\mathcal{T}_j$, $j\in\{1,2,3\}$,  defined in \eqref{Operators} are bounded from $(L^{\infty}_s)^3$ to $L^{\infty}_{s+\gamma}$, i.e.
\begin{equation*}
\|\mathcal{T}_j(f,g,h)\|_{L^\infty_s}\lesssim_s\|f\|_{L^\infty_s}\|g\|_{L^\infty_s}\|h\|_{L^\infty_s}.
\end{equation*}
\end{proposition}

\begin{proposition} \label{PropL2}
Assume that $\omega(p)=|p|^2$ and $s>1/2$. Then the operators $\mathcal{T}_j$, $j\in\{1,2,3\}$,  defined in \eqref{Operators} are bounded from $(L^2_s)^3$ to $L^{2}_{s}$, i.e.
\begin{equation*}
\|\mathcal{T}_j(f,g,h)\|_{L^2_s}\lesssim_s\|f\|_{L^2_s}\|g\|_{L^2_s}\|h\|_{L^2_s}.
\end{equation*}
\end{proposition}

Propositions \ref{PropLinfty} and \ref{PropL2} and Theorems \ref{theorem:main} and \ref{theorem:main2} are proved in the next three sections. We conclude this section with several remarks:

\begin{remark}
The above theorems are optimal in terms of the exponent $s$ since it is not possible to define the operators $\mathcal{T}_j$ if $\omega(p) = |p|^2$ and the input functions have general tails decaying like $|p|^{-2}$. The two  theorems are also nearly critical since the spaces $L^{\infty}_s$, $s>2$, and $L^2_s$, $s>1/2$, are nearly critical with respect to the scaling~\eqref{scaling} of the equation. 
\end{remark}

\begin{remark}
We are working in dimension $d=3$ mostly for the sake of concreteness. Similar theorems hold in any dimension $d\geq 2$, with the corresponding ranges of exponents $s>d-1$ for the $L^\infty_s$ local well-posedness theory, and  $s>(d-2)/2$ for the $L^2_s$ local well-posedness theory. 
\end{remark}

\begin{remark} As long as $\omega(p) \sim |p|^2$ for $|p| \to \infty$, the stationary solutions~\eqref{stationary} are on the borderline of the local well-posedness theory, since they belong to the scale-invariant space $L^\infty_2$. Notice that this only occurs in dimension $3$.
\end{remark}

\begin{remark}
It is probably possible to prove nearly critical $L^2_s$ local well-posedness theorems for more general radial dispersion relations $\omega$. However, one would likely have to assume some additional curvature assumptions on $\omega$, expressed in terms of bounds on the second derivative $\Omega''$, in order to be able to run $TT^\ast$ arguments for Radon transforms, as in section \ref{Proof4}. For simplicity, we consider here only the Schr\"{o}dinger case $\omega(p)=|p|^2$. 
\end{remark}

\begin{remark} It would be possible to prove identical local well-posedness results for the more general equation $\partial_t f = a_1\mathcal{T}_1(f,f,f)+a_2\mathcal{T}_2(f,f,f)+a_3\mathcal{T}_3(f,f,f)$, but the conservation law and the positivity of the solution would be lost.
\end{remark}

\begin{remark} The solution given by Theorem~\ref{theorem:main} has the property that
$$
f(t,p) - f_0(p) \in C^1([0,T),L^\infty_{s+\gamma})
$$
for some $\gamma>0$ (as a consequence of Proposition \ref{PropLinfty}). This means that the decay at $\infty$ of $f(t)$ is exactly the same as that of the data $f_0$. This should of course be contrasted with the cases of the classical Boltzmann equation \cite{AlonsoGamba:2013:ANA,BobylevGamba:2006:BEF,GambaPanferov:2004:OTB,GambaPanferovVillani:2009:UMB} and the quantum Boltzmann equation for bosons at very low temperature \cite{AlonsoGambaBinh} (this is also the weak turbulence kinetic equation for anharmonic crystal lattices), for which the decay of the solution is immediately improved.
\end{remark}

\begin{remark} For some data one can prove additional properties of the solution, such as conservation laws. See section \ref{Further}.
\end{remark}

\section{Proof of Proposition~\ref{PropLinfty}: $L^\infty_s (s>2)$ boundedness of $\mathcal{T}_j$}

Notice that, in the case $\omega(p) = |p|^2$, the desired bound follows easily from the formulation~\eqref{1D4wave}. The aim of this section is to explore the case of more general dispersion relations~$\omega$, for which no such simple representation of the collision operator is available.

\subsection{Boundedness of $\mathcal{T}_1$}\label{Sec:T1}

\begin{proposition}\label{proposition:T1}
For $s>2$ and $0\le \gamma < \min(s-2,1) $, and under Assumption~\ref{Assump1}, the operator~$\mathcal{T}_1$ is bounded from $(L^{\infty}_s)^3$ to $L^{\infty}_{s+\gamma}$.
\end{proposition}

\begin{proof} \underline{Step 1: first reduction.} It suffices to prove that the following integral is bounded:
\begin{equation}
\label{proposition:T1:E1}
\mathfrak{J} \ := \ \sup_{p}\iiint_{\mathbb{R}^9}\frac{\langle p\rangle^{s+\gamma}}{\langle p_1\rangle^s\langle p_2\rangle^s\langle p_3\rangle^s}\delta(p+p_1-p_2-p_3)\delta(\omega+\omega_1-\omega_2-\omega_3)\,dp_1\,dp_2\,dp_3.
\end{equation}

Since in the above integral $\omega(p) \le \omega(p_2)+\omega(p_3)$, then either $\omega(p) \le 2 \omega(p_2)$ or $\omega(p) \le 2 \omega(p_3)$. Suppose that $\omega(p)  \le 2\omega(p_3)$, which implies, by Assumption~\ref{Assump1}, that $\langle p\rangle \lesssim \langle p_3\rangle$. We then infer that
\begin{equation*}
\mathfrak{J} \ \lesssim  \ \sup_{p}\iiint_{\mathbb{R}^9}\frac{\langle p\rangle^{\gamma}}{\langle p_1\rangle^s\langle p_2\rangle^s}\delta(p+p_1-p_2-p_3)\delta(\omega+\omega_1-\omega_2-\omega_3) \, dp_1 \, dp_2\,dp_3.
\end{equation*}
Integrating out the $p_3$ variable results in
\begin{equation}\label{proposition:T1:E1a}
\mathfrak{J} \ \lesssim  \ \sup_{p}\iint_{\mathbb{R}^6}\frac{\langle p\rangle^{\gamma}}{\langle p_1\rangle^s\langle p_2\rangle^s}\delta(\omega+\omega_1-\omega_2-\omega(p+p_1-p_2))\,dp_1\,dp_2.
\end{equation}
Let us now set $z=p_2$ and define the resonant manifold $\mathcal{S}_{p,p_1}$ to be the zero set of
 \begin{equation}\label{proposition:T2L2:E3}
\mathfrak{G}(z) \ := \ \omega(p+p_1-z) + \omega (z) - \omega(p)-\omega(p_1)=0,\end{equation}
which leads to the following representation of the right hand side of \eqref{proposition:T1:E1a}, (see \cite{OsherFedkiw:2003:LSM}, section 1.5)
 \begin{equation}\label{proposition:T2L2:E3a}
\mathfrak{J} \ \lesssim  \ \sup_{p}\int_{\mathbb{R}^3}\left(\int_{\mathcal{S}_{p,p_1}}\frac{\langle p\rangle^{\gamma}}{\langle p_1\rangle^{s} \langle z\rangle^{s}|\nabla_z \mathfrak{G}(z)|}d\mu(z)\right)dp_1,
\end{equation}
where $\mu$ is the surface measure on $\mathcal{S}_{p,p_1}$.

\bigskip

\noindent \underline{Step 2: parameterizing the resonant manifold.}
Setting $p+p_1=\rho$, we now parameterize the resonant manifold $\mathcal{S}_{p,p_1}$. In order to do this, we compute the derivative of $\mathfrak{G}$ 
$$\nabla_z \mathfrak{G} = \frac{z - \rho}{|z-\rho|} \Omega'(|\rho -z|) + \frac{z}{|z|} \Omega'(|z|).$$
In particular, let $q$ be any vector orthogonal to $\rho$ i.e. $\rho \cdot q =0$. The directional derivative of $ \mathfrak{G}$ in the direction of $q$, with  $z=\alpha \rho +q, \alpha\in\mathbb{R}$, satisfies 
\begin{equation*} q \cdot \nabla_z \mathfrak{G} = |q|^2 \Big[ \frac{\Omega'(|\rho -z|)}{|\rho-z|} + \frac{\Omega'(|z|) }{|z|}\Big]   > 0,\end{equation*}
which means that $\mathfrak{G}(z)$ is strictly increasing in any direction that is orthogonal to $\rho$. This proves that the intersection between the surface $\mathcal{S}_{p,p_1}$ and the  plane $$\mathcal{P}_\alpha= \Big\{ \alpha \rho + q, \rho\cdot q = 0\Big\}$$ 
 is either empty or the circle centered at $\alpha \rho$ and of a finite radius $r_\alpha$, for  $\alpha \in \mathbb{R}$.

As a consequence, we can parametrize $\mathcal{S}_{p,p_1}$ as  follows. Let $\rho^\perp$ be the vector orthogonal to both $\rho$ and a fixed vector $e$ of $\mathbb{R}^3$ and let $e_\theta$ be the unit vector in $ \mathcal{P}_0=\{ \rho\cdot q = 0 \}$ such that the angle between $\rho^\perp$ and $e_\theta$ is $\theta$. We parameterize $\mathcal{S}_{p,p_1}$ by (cf. \cite{ToanBinh})
\begin{equation}\label{Para}
\Big\{ z  = \alpha \rho + r_\alpha e_\theta ~:~ \theta \in [0,2\pi], ~\alpha \in A_{p,p_1} \Big\},
\end{equation}
where $A_{p,p_1}$ is the set of $\alpha$ for which a solution to $\mathfrak{G}(z) = 0$ exists. 

We can think of $\mathfrak G$ as a function of  $\alpha$ and $r$: $\mathfrak G = \mathfrak G(r,\alpha)$. We just saw that $\partial_r \mathfrak G > 0$ for $r>0$. Therefore, by the implicit function theorem, the zero set of $\mathfrak{G}$ can be parameterized as
$$
\{ (\alpha, r = r_\alpha), \, \alpha \in  A_{p,p_1}\},
$$
where $\alpha \mapsto r_\alpha$ is a smooth function on $A_{p,p_1}$ vanishing on its boundary. 

Next, we have by definition that $\mathfrak{G}(z_\alpha) =0$ for all $\alpha$ and therefore, keeping $\theta$ fixed,
\begin{equation}
\label{colibri2}
\begin{aligned}
 0 &= \partial_\alpha z_\alpha \cdot \nabla_z \mathfrak{G} =\partial_\alpha z_\alpha \cdot \left(\frac{z_\alpha -\rho}{|z_\alpha -\rho|}\Omega'(|z_\alpha -\rho|)+\frac{z_\alpha}{|z_\alpha|}\Omega'(|z_\alpha|)\right)\\
   &=\partial_\alpha z_\alpha \cdot \left(\frac{z_\alpha}{|z_\alpha -\rho|}\Omega'(|z_\alpha -\rho|)+\frac{z_\alpha}{|z_\alpha|}\Omega'(|z_\alpha|)\right) - \partial_\alpha z_\alpha \cdot \frac{\rho}{|z_\alpha -\rho|}\Omega'(|z_\alpha -\rho|)\\
 & = \frac12 \partial_\alpha |z_\alpha|^2  \Big[ \frac{\Omega'(|\rho -z_\alpha|)}{|\rho-z_\alpha|} + \frac{\Omega'(|z_\alpha|) }{|z_\alpha|}\Big]  - |\rho|^2 \frac{\Omega'(|\rho -z_\alpha|)}{|\rho-z_\alpha|}.
 \end{aligned}\end{equation}
Therefore,
\begin{equation}
\label{colibri}
\partial_\alpha|z_\alpha|^2  = 2
 \frac{ \frac{\Omega'(|\rho -z_\alpha|)}{|\rho-z_\alpha|}|\rho|^2 }{ \frac{\Omega'(|\rho -z_\alpha|)}{|\rho-z_\alpha|} + \frac{\Omega'(|z_\alpha|) }{|z_\alpha|}}.
\end{equation}
This implies in particular that $\alpha \mapsto |z_\alpha|$ is increasing on $A_{p,p_1}$. Defining $r$ to be zero on the complement of $A_{p,p_1}$, we get that $\alpha \mapsto |z_\alpha|$ is an increasing function on $\mathbb{R}$; therefore, the change of coordinates $\alpha \to |z_\alpha|$ is well-defined.

\bigskip
\noindent \underline{Step 3: the surface measure on the resonant manifold} Since $\partial_\theta e_\theta$ is orthogonal to both $\rho$ and $e_\theta$,  we compute the surface area 
\begin{equation}\label{dS}\begin{aligned}
 d \mu (z) &= |\partial_\alpha z \times \partial_\theta z | d\alpha d\theta  = \Big |(\rho+\partial_\alpha r_\alpha e_\theta) \times  r_\alpha \partial_\theta e_\theta \Big | d\alpha d\theta
 \\
 &=
 \sqrt{|\rho|^2  r_\alpha ^2+ \frac14| \partial_\alpha ( r_\alpha^2)|^2}d\alpha d\theta .
 \end{aligned}\end{equation}
Using $|z|^2 = \alpha^2 |\rho|^2 + r_\alpha^2$, we learn from the last line of~\eqref{colibri2} that
 \begin{equation}\label{proposition:T2L2:E8}
\begin{aligned}
  \partial_\alpha  r_\alpha^2  &
 &= 
2|\rho|^2 \frac{\alpha   \frac{\Omega'(|z_\alpha|) }{|z_\alpha|} +(\alpha-1)  \frac{\Omega'(|\rho -z_\alpha|)}{|\rho-z_\alpha|}}{\frac{\Omega'(|\rho -z_\alpha|)}{|\rho-z_\alpha|} + \frac{\Omega'(z_\alpha) }{|z_\alpha|}}.
 \end{aligned}\end{equation}
Now, let us compute $|\nabla_z\mathfrak{G}|$ under the new parameterization:
\begin{equation*}
\begin{aligned}
 |\nabla_z \mathfrak{G}|^2 \ =  &\left|\frac{z_\alpha}{|z_\alpha|}\Omega'(|z_\alpha|)+\frac{z_\alpha -\rho}{|z_\alpha -\rho|}\Omega'(|z_\alpha -\rho|)\right|^2\\
 = & \left|\frac{\alpha \rho+q}{|z_\alpha|}\Omega'(|z_\alpha|)+\frac{(\alpha-1)\rho+q}{|z_\alpha -\rho|}\Omega'(|z_\alpha -\rho|)\right|^2\\
 \ =  &\  |\rho|^2\left[\alpha   \frac{\Omega'(|z_\alpha|) }{|z_\alpha|} +(\alpha-1)  \frac{\Omega'(|\rho -z_\alpha|)}{|\rho-z_\alpha|}\right]^2 + r_\alpha^2 \left[\frac{\Omega'(|\rho -z_\alpha|)}{|\rho-z_\alpha|} + \frac{\Omega'(|z_\alpha|) }{|z_\alpha|}\right]^2.
\end{aligned}
\end{equation*}
In addition to \eqref{proposition:T2L2:E8}, this implies that
\begin{equation}\label{proposition:T2L2:E9}
\begin{aligned}
 |\nabla_z \mathfrak{G}|^2 \ =  &\  \frac{\left|\partial_\alpha r_\alpha^2\right|^2}{4|\rho|^2}\left[\frac{\Omega'(|\rho -z_\alpha|)}{|\rho-z_\alpha|} + \frac{\Omega'(|z_\alpha|) }{|z_\alpha|}\right]^2 + r_\alpha^2 \left[\frac{\Omega'(|\rho -z_\alpha|)}{|\rho-z_\alpha|} + \frac{\Omega'(|z_\alpha|) }{|z_\alpha|}\right]^2.
\end{aligned}
\end{equation}
Therefore
 \begin{equation}\label{proposition:T2L2:E9a}\begin{aligned}
\frac{ d \mu (z)}{ |\nabla_z \mathfrak{G}|} &= \frac{|\rho|}{\frac{\Omega'(|\rho -z_\alpha|)}{|\rho-z_\alpha|} + \frac{\Omega'(|z_\alpha|) }{|z_\alpha|}}\, d\alpha\, d\theta.
 \end{aligned}\end{equation}
Introduce the variable $u = |z_\alpha | = \sqrt{\alpha^2 |\rho|^2 +r_\alpha^2}$ as explained in Step 2; by~\eqref{colibri} we get
$$
\frac{ d \mu (z)}{ |\nabla_z \mathfrak{G}|} = \frac{|\rho-z_\alpha|}{\Omega'(|\rho -z_\alpha|) |\rho|} u\,du \, d\theta.
$$
By Assumption \ref{Assump1}, $\frac{|\rho-z_\alpha|}{\Omega'(|\rho -z_\alpha|)}\lesssim 1,$ and therefore
 \begin{equation}\label{proposition:T2L2:E9b}\begin{aligned}
\frac{ d \mu (z)}{ |\nabla_z \mathfrak{G}|} &\lesssim \frac{u}{|\rho|}du d\theta.
 \end{aligned}\end{equation}

\bigskip

\noindent \underline{Step 4: finiteness of the integral.}
Adopting the coordinates defined above and using \eqref{proposition:T2L2:E9b} yields
\begin{equation*}
\mathfrak{J} \ \lesssim  \ \sup_{p}\int_{\mathbb{R}^3}\frac{\langle p\rangle^{\gamma}}{\langle p_1\rangle^s}\int_0^{\infty}\int_0^{2\pi}{\langle u\rangle^{-s}\frac{|u|}{|\rho|}}\,d\theta\, du\, dp_1,
\end{equation*}
Changing variables $p_1\to \rho=p+p_1$, this becomes
\begin{equation*}
\mathfrak{J} \  \lesssim \ \sup_{p}\int_{\mathbb{R}^3}\frac{\langle p\rangle^{\gamma}}{\langle \rho -p \rangle^s}\int_0^{\infty}\int_0^{2\pi}{\langle u\rangle^{-s}\frac{|u|}{|\rho|}}\,d\theta  \, du\, d\rho.
\end{equation*}
Performing the integrations in $z$ and $\theta$, this leads to
\begin{equation*}
\mathfrak{J} \ \lesssim\sup_{p}\int_{\mathbb{R}^3}\frac{\langle p\rangle^{\gamma}}{\langle \rho -p \rangle^s}\frac{1}{|\rho|}d\rho.
\end{equation*}
Writing $\rho=|\rho|\sigma	$ where $\sigma	\in\mathbb{S}^2$ and using the inequality\footnote{In order to prove this inequality, simply observe that
$$\int_{\mathbb{S}^2}\frac{1}{\langle A+r\sigma	 \rangle^{s}} d\sigma = \int_0^{\pi} \frac{\sin \phi}{(|A|^2 + r^2 - 2|A|r \cos \phi + 1)^{s/2}}\,d\phi.$$ The main contribution is
$$
\int_0^{\pi/2} \frac{\sin \phi}{(|A|^2 + r^2 - 2|A|r \cos \phi + 1)^{s/2}}\,d\phi = \int_0^1 \frac{dt}{(\langle |A|-r \rangle^2 + 2 |A|r t )^{s/2}}.
$$}
\begin{equation}
\label{Inequality}
\int_{\mathbb{S}^2}\frac{1}{\langle A+r\sigma	 \rangle^{s}} d\sigma \lesssim \langle |A|-r\rangle^{2-s}\langle r\rangle^{-2} \qquad \forall A\in\mathbb{R}^3, r>0, s>2,
\end{equation}
we get 
\begin{equation*}
\mathfrak{J} \ \lesssim \ \sup_{p}\int_{0}^\infty\frac{|\rho|\langle p\rangle^{\gamma}}{\langle |\rho| - |p| \rangle^{s-2}\langle \rho\rangle^2}d|\rho|,
\end{equation*}
which is bounded when $s>2$ and $0\le \gamma <\min(s-2,1)$.
\end{proof}

\subsection{Boundedness of $\mathcal{T}_2$}\label{Sec:T2}

\begin{proposition}\label{proposition:T2}
For $s>2$ and $0\le \gamma < s-2$, and under Assumption~\ref{Assump1}, the operator~$\mathcal{T}_2$ is bounded from $(L^{\infty}_s)^3$ to $L^{\infty}_{s+\gamma}$.
\end{proposition}

\begin{proof}\underline{Step 1: reduction to the boundedness of $\mathcal{Q}_1$.} Defining
 \begin{equation}\label{proposition:T2L2:E2}
\mathcal{Q}_1(g,h)(p) \ = \ \iint_{\mathbb{R}^{6}}\delta(\omega +\omega_1-\omega_2-\omega(p+p_1-p_2))g(p_2)h(p+p_1-p_2) \, dp_1 \, dp_2,
\end{equation}
it suffices to prove that
$$\|\mathcal{Q}_1(g,h)\|_{L^\infty_\gamma}\lesssim \ \|g\|_{L^\infty_{s}}\|h\|_{L^\infty_{s}},$$
Taking the $L^\infty$ norm of $\mathcal{Q}_1(g,h)$, we obtain
 \begin{equation}\label{proposition:T2L2:E2}
 \begin{aligned}
\|\mathcal{Q}_1(g,h)\|_{L^\infty_\gamma}  &  \leq  \sup_{p\in\mathbb{R}^3}\|g\|_{L^\infty_s}\|h\|_{L^\infty_s}\times\\
& \times \iint_{\mathbb{R}^{6}}\delta(\omega+\omega_1-\omega_2-\omega(p+p_1-p_2))\langle p_2\rangle^{-s} \langle p+p_1-p_2\rangle^{-s}\langle p\rangle^{\gamma}\,dp_1 \, dp_2. 
\end{aligned}
\end{equation}

\bigskip

\noindent \underline{Step 2: upper bound  on $|p|$.} Keeping the notations of Section~\ref{Sec:T1}, we deduce from the inequalities
$\omega(p) + \omega(p_1) = \omega(p_2) + \omega(p_3)$ and $p + p_1 = p_2 + p_3$ that
$$
\omega(p) \leq \omega(z) + \omega(\rho-z).
$$
We now use Assumption~\ref{Assump1} to bound
$$
\omega(p) \leq \omega(z) + \omega(\rho-z) \leq 2 \Omega \left( \max(|\rho|,|\rho-z|) \right) \leq \Omega(c_2 \max(|\rho|,|\rho-z|)).
$$
Since $\Omega$ is increasing, this implies that
$$
|p| \lesssim |\rho| + |\rho-z|.
$$

\bigskip

\noindent \underline{Step 3: parameterizing the integral.} Adopting the same parameterization as in Section~\ref{Sec:T1}, it appears that \eqref{proposition:T2L2:E2} would follow from a bound on 
 \begin{equation}\label{proposition:T2L2:E3a}
\sup_{p\in\mathbb{R}^3} \int_{\mathbb{R}^{3}}\left(\int_{\mathcal{S}_{p,p_1}}\frac{\langle p\rangle^{\gamma}\langle z\rangle^{-s} \langle p+p_1-z\rangle^{-s}}{|\nabla_z \mathfrak{G}(z)|}d\mu(z)\right)dp_1.
\end{equation}
By the parametrization \eqref{Para} and Step 2, matters reduce to bounding
\begin{equation*}
\sup_{p\in\mathbb{R}^3} \int_{\mathbb{R}^3}\int_0^\infty\int_0^{2\pi} \mathbf{1}_{|p| \lesssim |\rho| + |\rho-z|} \frac{\langle p\rangle^{\gamma} |z|}{\langle z\rangle^{s}\langle \rho-z\rangle^{s}|\rho|} \, d\theta \, d|z| \, dp_1,
\end{equation*}
where $\mathbf{1}_{|p| \lesssim |\rho| + |\rho-z|}$ is the characteristic function of $\{|p| \lesssim |\rho| + |\rho-z| \}$. On the one hand, integrating in $\theta$ is harmless; and on the other hand, in the above integral, either $|p| \lesssim |z|$ or $|p| \lesssim |\rho-z|$. Therefore, it suffices to bound
\begin{equation*}
\sup_{p\in\mathbb{R}^3} \int_{\mathbb{R}^3}\int_0^\infty  \frac{ |z|}{\langle z\rangle^{s_1}\langle \rho-z\rangle^{s_2}|\rho|} \, d|z| \, dp_1,
\end{equation*}
where $s_1,s_2 > 2$. Changing variables from $p_1$ to $\rho$, this becomes
 \begin{equation*}
\sup_{p\in\mathbb{R}^3} \int_{\mathbb{R}^3}\int_0^{\infty}\frac{|z|}{\langle z\rangle^{s_1}\langle \rho-z\rangle^{s_2}|\rho|}d|z| d\rho .
\end{equation*}
Writing $\rho$ as $|\rho|\omega$ and using \eqref{Inequality}, we obtain  that the above is bounded by
 \begin{equation}\label{proposition:T2L2:E12}
 \begin{aligned}
\sup_{p\in\mathbb{R}^3} \int_0^{\infty}\int_0^{\infty}\frac{|z||\rho|}{\langle z\rangle^{s_1}\langle |\rho|-|z|\rangle^{s_2-2}\langle \rho \rangle^2}d|z| d|\rho| ,
\end{aligned}
\end{equation}
which is finite for $s_1,s_2>2$. This is the desired result!
\end{proof}

\subsection{Boundedness of $\mathcal{T}_3$}\label{Sec:T3}

\begin{proposition}\label{proposition:T3}
For $s>2$ and $0\le \gamma < s-2 $, and under Assumption~\ref{Assump1}, the operator~$\mathcal{T}_3$ is bounded from $(L^{\infty}_s)^3$ to $L^{\infty}_{s+\gamma}$.
\end{proposition}

\begin{proof} Defining
\begin{equation}\label{Q2}
\mathcal{Q}_2(g,h) \  = \  \int_{\mathbb{R}^6}\delta(\omega +\omega_1 -\omega_2 -\omega(p+p_1-p_2)) g(p_1)h(p_2) \, dp_1 \,dp_2,
\end{equation}
it suffices to show that
$$\|\mathcal{Q}_2(g,h)\|_{L^\infty_\gamma} \lesssim \|g\|_{L^\infty_{s}}\|h\|_{L^\infty_{s}}.$$
Similarly to Section \ref{Sec:T2}, we set $\rho=p+p_1$, and define $\mathfrak{G}$ and $\mathcal{S}_{p,p_1}$. 
Proceeding as in the proof of Proposition~\ref{proposition:T1}, it suffices to prove the boundedness of
\begin{equation}\label{Sec:T3:E9}
\begin{aligned}
 \mathcal{J} =  \ \sup_{p\in\mathbb{R}^3} \int_{\mathbb{R}^3}\langle p_1\rangle^{-s}\langle p\rangle^{\gamma} \int_0^{\infty}\int_0^{2\pi}{\langle z\rangle^{-s}\frac{|z|}{|\rho|}}\, d\theta \, d|z| \, dp_1 .
\end{aligned}
\end{equation}
Setting $\rho = p + p_1$, this becomes
\begin{equation}\label{Sec:T3:E9aa}
\begin{aligned}
 \mathcal{J} \ & =  \ \sup_{p\in\mathbb{R}^3} \int_{\mathbb{R}^3}\langle \rho - p\rangle^{-s}\langle p\rangle^{\gamma} \int_0^{\infty}\int_0^{2\pi}{\langle z\rangle^{-s}\frac{|z|}{|\rho|}}\,d\theta \, d|z| \, d\rho \\
 & \lesssim  \ \sup_{p\in\mathbb{R}^3} \int_{\mathbb{R}^3}\int_0^{\infty}\langle \rho - p\rangle^{-s}\langle p\rangle^{\gamma}{\langle z\rangle^{-s}\frac{|z|}{|\rho|}} \, d|z| \, d\rho \\
 & \lesssim  \ \sup_{p\in\mathbb{R}^3} \int_{\mathbb{R}^3}\langle \rho - p\rangle^{-s}\langle p\rangle^{\gamma}{\frac{1}{|\rho|}} \, d\rho ,
\end{aligned}
\end{equation}
 where the last inequality is due to the fact that $s>2$.
 
 Writing $\rho$ as $|\rho|\omega$ and using \eqref{Inequality}, we obtain

\begin{equation}\label{Sec:T3:E9aaa}
\begin{aligned}
 \mathcal{J}
 & \lesssim  \ \sup_{p\in\mathbb{R}^3} \int_0^{\infty}{\frac{\langle p\rangle^{\gamma}|\rho|}{\langle\rho\rangle^2\langle |\rho| - |p|\rangle^{s-2}}}d|\rho| ,
\end{aligned}
\end{equation}
which is bounded when $s>2$ and $0\le \gamma< s-2$.
\end{proof}

\section{Proof of Proposition~\ref{PropL2}: $L^2_s (s>1/2)$ boundedness of $\mathcal{T}_j$}\label{Proof4}

In this section we assume that $\omega(p)=|p|^2$ and prove the $L^2_s$ bounds in Proposition~\ref{PropL2}. One can think of the operators $\mathcal{T}_j$ as bilinear and trilinear operators defined by integrating along moving surfaces in Euclidean spaces. Such operators are called Radon transforms, and their boundedness properties have been studied extensively in Harmonic Analysis (see for example the classical papers \cite{Seeger:RTA:1998,sogge1985averages,SoggeStein:AOH:1990}).

One of the main ideas in the study of Radon transforms on Euclidean spaces is the use of $TT^\ast$ arguments. We adapt this technique in our setting to bound the trilinear operators $\mathcal{T}_j$. We remark that  $TT^\ast$ arguments are usually optimal if one uses $L^2$ based spaces; this is the main reason for choosing the spaces $L^2_s$ as the local well-posedness spaces in Theorem \ref{theorem:main2}.

\subsection{The operator $\mathcal{T}_1$}

We consider first the trilinear operator $\mathcal{T}_1$ and we prove the following:

\begin{lemma}\label{Lem1}
If $s>1/2$ and $\mathcal{T}_1$ is defined as in \eqref{Operators} then
\begin{equation}\label{ma5}
\|\mathcal{T}_1\|_{L^2_s\times L^2_s\times L^2_s\to L^2_s}\lesssim_s 1.
\end{equation}
\end{lemma}

\begin{proof} We adapt an argument from \cite{buckmaster2016analysis}. We start from the identity
\begin{equation*}
\delta(q)=\frac{1}{2\pi}\int_{\mathbb{R}}e^{iq\xi}\,d\xi.
\end{equation*}
For simplicity of notation, let $Q:=\mathcal{T}_1[f,g,h]$. We have
\begin{equation*}
\begin{split}
Q(p)&=C\int_{(\mathbb{R}^d)^3\times\mathbb{R}\times\mathbb{R}}e^{iy\cdot(p+p_1-p_2-p_3)}e^{it(\omega(p)+\omega(p_1)-\omega(p_2)-\omega(p_3))}f(p_1)g(p_2)h(p_3)\,dp_1 dp_2 dp_3 dt dy\\
&=C\int_{\mathbb{R}\times\mathbb{R}}e^{iy\cdot p}e^{it\omega(p)}\overline{L\overline{f}}(y,t)Lg(y,t)Lh(y,t)\,dt dy,
\end{split}
\end{equation*}
where
\begin{equation}\label{ma7}
La(x,t):=\int_{\mathbb{R}^d}a(q)e^{-iq\cdot x}e^{-i\omega(q)t}\,dq.
\end{equation}
Therefore, with $G(y,t):=\overline{L\overline{f}}(y,t)Lg(y,t)Lh(y,t)$,
\begin{equation*}
\|\langle p\rangle ^sQ(p)\|_{L^2}\lesssim \Big\|\int_{\mathbb{R}}\mathcal{F}^{-1}(G)(p,t)e^{it\omega(p)}\langle p\rangle ^s\,dt\Big\|_{L^2}\lesssim \int_{\mathbb{R}}\|G(.,t)\|_{H^s}\,dt,
\end{equation*}
where $H^s$ denote the usual Sobolev spaces on $\mathbb{R}^3$. Notice that, for any $t\in\mathbb{R}$,
\begin{equation*}
\|G(.,t)\|_{H^s}\lesssim \sum_{\{a,b,c\}=\{f,g,h\}}\|La(.,t)\|_{H^s}\|Lb(.,t)\|_{L^\infty}\|Lc(.,t)\|_{L^\infty}.
\end{equation*}
Moreover, for any $a\in\{f,g,h\}$,
\begin{equation*}
\sup_{t\in\mathbb{R}}\|La(.,t)\|_{H^s}\lesssim \|a\|_{L^2_s}.
\end{equation*}
In view of the last three inequalities, for \eqref{ma5} it suffices to prove the linear estimates
\begin{equation}\label{ma10}
\Big[\int_{\mathbb{R}}\|Lb(.,t)\|_{L^\infty}^2\,dt\Big]^{1/2}\lesssim_s\|b\|_{L^2_s}
\end{equation}
for any $s>1/2$ and $b\in L^2_s$.

The estimates \eqref{ma10} are Strichartz-type linear estimates. To prove them we use a $TT^\ast$-type argument. We may assume that $\|b\|_{L^2_s}=1$ and $b(p)=h(p)\langle p\rangle ^{-s}$, $\|h\|_{L^2}=1$. For \eqref{ma10} it suffices to show that
\begin{equation*}
\Big|\int_{\mathbb{R}^3\times\mathbb{R}}Lb(x,t)F(x,t)\,dxdt\Big|\lesssim_s 1
\end{equation*}
provided that $\big[\int_{\mathbb{R}}\|F(.,t)\|_{L^1}^2\,dt\big]^{1/2}\lesssim 1$. Using \eqref{ma7}, this is equivalent to proving that
\begin{equation*}
\Big\|\langle p\rangle ^{-s}\int_{\mathbb{R}^3\times\mathbb{R}}F(x,t)e^{-ip\cdot x}e^{-it\omega(p)}\,dxdt\Big\|_{L^2}\lesssim_s 1,
\end{equation*}
where the $L^2$ norm is taken in the $p$ variable. Expanding the $L^2$ norm in $p$, this is equivalent to showing that
\begin{equation}\label{ma11}
\Big|\int_{\mathbb{R}^3}\int_{\mathbb{R}^3\times\mathbb{R}}\int_{\mathbb{R}^3\times\mathbb{R}}\langle p\rangle ^{-2s}F(x,t)e^{-ip\cdot x}e^{-it\omega(p)}\overline{F(x',t')}e^{ip\cdot x'}e^{it'\omega(p)}\,dxdtdx'dt'dp\Big|\lesssim_s 1.
\end{equation} 
Let
\begin{equation}\label{ma12}
K(y,t):=\int_{\mathbb{R}^3}\langle p\rangle ^{-2s}e^{-ip\cdot y}e^{-it\omega(p)}\,dp,
\end{equation}
so the left-hand side of \eqref{ma11} is bounded by
\begin{equation*}
C\int_{\mathbb{R}^3\times\mathbb{R}}\int_{\mathbb{R}^3\times\mathbb{R}}|F(x,t)||F(x',t')||K(x-x',t-t')|\,dxdtdx'dt'.
\end{equation*}
Since $\big[\int_{\mathbb{R}}\|F(.,t)\|_{L^1}^2\,dt\big]^{1/2}\lesssim 1$, and recalling that $s>1/2$, for \eqref{ma10} it suffices to prove that there is $\delta=\delta(s)>0$ such that
\begin{equation}\label{ma13}
|K(y,t)|\lesssim_\delta |t|^{-1+\delta}\langle t\rangle^{-2\delta},
\end{equation}
for any $(y,t)\in\mathbb{R}^3\times\mathbb{R}$. Recalling that $\omega(p)=|p|^2$, this is a standard dispersive bound on the kernel of the Schr\"{o}dinger evolution and can be proved by oscillatory integral estimates.
\end{proof}

\subsection{The operator $\mathcal{T}_2$}

Notice that $\mathcal{T}_2(f,g,h)=f\cdot Q_2(g,h)$ where, by definition,
\begin{equation}\label{bil1}
\begin{split}
Q_2(F,G)(p)&:=\int_{\mathbb{R}^3\times\mathbb{R}^3\times\mathbb{R}^3}\delta(p+p_1-x-y)\delta(\omega(p)+\omega(p_1)-\omega(x)-\omega(y))F(x)G(y)\,dxdy dp_1\\
&=\int_{\mathbb{R}^3\times\mathbb{R}^3}\delta(\omega(p)+\omega(x+y-p)-\omega(x)-\omega(y))F(x)G(y)\,dxdy.
\end{split}
\end{equation}

The boundedness of the operator $\mathcal{T}_2$ follows from the following lemma:

\begin{lemma}\label{prop1}
If $\omega(x)=|x|^2$ and $s>1/2$ then
\begin{equation}\label{bil3}
\|Q_2(F,G)\|_{L^\infty}\lesssim_s \|F\|_{L^2_s}\|G\|_{L^2_s}.
\end{equation}
\end{lemma}

\begin{proof} We replace the $\delta_0$ function with a smooth version. More precisely, we fix a smooth even function $\psi:\mathbb{R}\to[0,\infty)$ supported in the interval $[-1,1]$ with $\int_{\mathbb{R}}\psi(t)\,dt=1$. For any $\varep\in(0,1]$ let $\psi_\varep(t):=(1/\varep)\psi(t/\varep)$. Since
\begin{equation*}
\omega(p)+\omega(x+y-p)-\omega(x)-\omega(y)=2(x-p)\cdot(y-p),
\end{equation*}
for \eqref{bil3} it suffices to prove that
\begin{equation}\label{bil4}
\Big|\int_{\mathbb{R}^3\times\mathbb{R}^3}\psi_\varep((x-p)\cdot(y-p))F(x)G(y)\,dxdy\Big|\lesssim_s \|F\|_{L^2_s}\|G\|_{L^2_s}
\end{equation}
for any $p\in\mathbb{R}^3$ and $\varep\in(0,1]$. We let
\begin{equation*}
f(x):=\langle x+p\rangle^sF(x+p),\qquad g(y):=\langle y+p\rangle^sG(y+p).
\end{equation*} 
After changes of variables, for \eqref{bil4} it suffices to prove that
\begin{equation}\label{bil5}
\Big|\int_{\mathbb{R}^3\times\mathbb{R}^3}\psi_\varep(x\cdot y)\frac{f(x)}{\langle x+p\rangle^s}\frac{g(y)}{\langle y+p\rangle^s}\,dxdy\Big|\lesssim_s \|f\|_{L^2}\|g\|_{L^2}.
\end{equation}
This is equivalent to proving $L^2$ boundedness of a linear operator, i.e.
\begin{equation}\label{bil6}
\|L_2g\|_{L^2}\lesssim_s\|g\|_{L^2}\qquad\text{ where }\qquad L_2g(x):=\int_{\mathbb{R}^3}\psi_\varep(x\cdot y)\frac{1}{\langle x+p\rangle^s}\frac{g(y)}{\langle y+p\rangle^s}\,dy,
\end{equation}
uniformly for any $p\in\mathbb{R}^3$ and $\varep\in(0,1]$. 

To prove \eqref{bil6} we use a $TT^\ast$-type argument. We may assume $g\geq 0$ and write
\begin{equation*}
\begin{split}
\|L_2g\|_{L^2}^2&=\int_{\mathbb{R}^3\times\mathbb{R}^3\times\mathbb{R}^3}\psi_\varep(x\cdot y)\psi_\varep(x\cdot y')\frac{1}{\langle x+p\rangle^{2s}}\frac{g(y)}{\langle y+p\rangle^s}\frac{g(y')}{\langle y'+p\rangle^s}\,dydy'dx\\
&=\int_{\mathbb{R}^3\times\mathbb{R}^3}K_s(y,y')\frac{g(y)}{\langle y+p\rangle^s}\frac{g(y')}{\langle y'+p\rangle^s}\,dydy',
\end{split}
\end{equation*}
where
\begin{equation*}
K_s(y,y')=K_{s,\varep,p}(y,y'):=\int_{\mathbb{R}^3}\psi_\varep(x\cdot y)\psi_\varep(x\cdot y')\frac{1}{\langle x+p\rangle^{2s}}\,dx.
\end{equation*}
Using Lemma \ref{lem4} (ii) below, we have
\begin{equation*}
|K_s(y,y')|\lesssim_s\frac{1}{|y||y'|}\Big(\frac{1}{|\widehat{y}-\widehat{y'}|}+\frac{1}{|\widehat{y}+\widehat{y'}|}\Big),
\end{equation*}
where $\widehat{x}:=x/|x|$ for any $x\in\mathbb{R}^3$. For \eqref{bil6} it suffices to prove that
\begin{equation}\label{bil7}
\Big|\int_{\mathbb{R}^3\times\mathbb{R}^3}\frac{1}{|y||y'|}\frac{1}{|\widehat{y}-\widehat{y'}|}\cdot\frac{g(y)}{\langle y+p\rangle^s}\frac{h(y')}{\langle y'+p'\rangle^s}\,dydy'\Big|\lesssim_s\|g\|_{L^2}\|h\|_{L^2},
\end{equation}
for any $g,h\in L^2(\mathbb{R}^3)$ and any $p,p'\in\mathbb{R}^3$.

For $\theta,\theta'\in\mathbb{S}^2$ let
\begin{equation*}
\widetilde{g}(\theta):=\Big[\int_0^\infty|g(r\theta)|^2r^2\,dr\Big]^{1/2},\qquad \widetilde{h}(\theta'):=\Big[\int_0^\infty|h(r\theta')|^2r^2\,dr\Big]^{1/2}.
\end{equation*}
We make the changes of variables $y=r\theta$ and $y'=r'\theta'$ in the integral in the left-hand side of \eqref{bil7}. Notice that
\begin{equation*}
\int_0^\infty\frac{g(r\theta)}{\langle r\theta+p\rangle^s}r\,dr\lesssim_s\widetilde{g}(\theta),\qquad\int_0^\infty\frac{h(r'\theta')}{\langle r'\theta'+p'\rangle^s}r'\,dr'\lesssim_s\widetilde{h}(\theta'),
\end{equation*}
using the Cauchy-Schwarz inequality and \eqref{bil19}. Thus the integral in the left-hand side of \eqref{bil7} is bounded by
\begin{equation*}
C_s\Big|\int_{\mathbb{S}^2\times\mathbb{S}^2}\frac{1}{|\theta-\theta'|}\cdot\widetilde{g}(\theta)\widetilde{h}(\theta')\,d\theta d\theta'\Big|.
\end{equation*}
Using Schur's lemma this is bounded by $\|\widetilde{g}\|_{L^2(\mathbb{S}^2)}\|\widetilde{h}\|_{L^2(\mathbb{S}^2)}$, and the desired estimates \eqref{bil7} follow. This completes the proof.
\end{proof}

We summarize below two technical estimates we used in the proof of Lemma \ref{prop1}.

\begin{lemma}\label{lem4}
(i) If $\theta\in\mathbb{S}^2$ and $p\in\mathbb{R}$ then
\begin{equation}\label{bil19}
\int_{\mathbb{R}}\frac{1}{\langle r\theta+p\rangle^{2s}}\,dr\lesssim_s 1.
\end{equation}

(ii) Assume that $\varep_1,\varep_2\in[0,1)$, $a,b\in\mathbb{R}$, $p\in\mathbb{R}^3$, $u,v\in\mathbb{S}^2$, and $s>1/2$. Then
\begin{equation}\label{bil20}
\int_{\mathbb{R}^3}\mathbf{1}_{[0,\varep_1]}(x\cdot v-a)\mathbf{1}_{[0,\varep_2]}(x\cdot w-b)\frac{1}{\langle x+p\rangle^{2s}}\,dx\lesssim_s \frac{\varep_1\varep_2}{|v-w|}+\frac{\varep_1\varep_2}{|v+w|}.
\end{equation}
\end{lemma}

\begin{proof} (i) By rotation invariance, we may assume $\omega=(1,0,0)$. The bound \eqref{bil19} is then implied by the easy estimate
\begin{equation}\label{bil21}
\sup_{q_1,q_2,q_3\in\mathbb{R}}\int_{\mathbb{R}}\frac{1}{\big[(r+q_1)^2+q_2^2+q_3^2+1\big]^s}\,dr\lesssim_s 1.
\end{equation}

(ii) We may assume $\varep_1\leq\varep_2$. By rotation invariance, we may assume $v=(1,0,0)$ and $w=(w_1,w_2,0)$. Clearly, $|w_2|\approx \min(|v-w|,|v+w|)$. Notice also that
\begin{equation*}
x\cdot w-b=x_1w_1+x_2w_2-b=x_2w_2-(b-aw_1)+(x_1-a)w_1.
\end{equation*}
Since $|w_1|\leq 1$, the integral in the left-hand side of \eqref{bil20} is bounded by
\begin{equation*}
\int_{\mathbb{R}^3}\mathbf{1}_{[0,\varep_1]}(x_1-a)\mathbf{1}_{[-4\varep_2,4\varep_2]}(x_2w_2-b')\frac{1}{\langle x+p\rangle^{2s}}\,dx.
\end{equation*}
The desired conclusion follows using \eqref{bil21} and integrating first the variable $x_3$.
\end{proof}

\subsection{The operator $\mathcal{T}_3$}

As in the previous subsection we notice that $\mathcal{T}_3(f,g,h)=f\cdot Q_3(g,h)$ where
\begin{equation}\label{bil30}
\begin{split}
Q_3(F,G)(p)&:=\int_{\mathbb{R}^3\times\mathbb{R}^3\times\mathbb{R}^3}\delta(p-p_3+x-y)\delta(\omega(p)-\omega(p_3)+\omega(x)-\omega(y))F(x)G(y)\,dxdy dp_3\\
&=\int_{\mathbb{R}^3\times\mathbb{R}^3}\delta(\omega(p)-\omega(x-y+p)+\omega(x)-\omega(y))F(x)G(y)\,dxdy.
\end{split}
\end{equation}

In view of the definitions, boundedness of $\mathcal{T}_3$ follows from the following lemma:

\begin{lemma}\label{prop2}
If $\omega(x)=|x|^2$ as before and $s>1/2$ then
\begin{equation}\label{bil31}
\|Q_3(F,G)\|_{L^\infty}\lesssim_s \|F\|_{L^2_s}\|G\|_{L^2_s}.
\end{equation}
\end{lemma}

\begin{proof} As before we replace $\delta$ with $\psi_\varep$ and notice that
\begin{equation*}
\omega(p)-\omega(x-y+p)+\omega(x)-\omega(y)=2(x-y)\cdot(y-p).
\end{equation*}
We let $f(x)=\langle x+p\rangle^sF(x+p)$ and $g(y)=\langle y+p\rangle^sG(y+p)$ as in the proof of Lemma \ref{prop1}. After changes of variables, for \eqref{bil31} it suffices to prove that
\begin{equation}\label{bil33}
\|L_3g\|_{L^2}\lesssim_s\|g\|_{L^2}\qquad\text{ where }\qquad L_3g(x):=\int_{\mathbb{R}^3}\psi_\varep((x-y)\cdot y)\frac{1}{\langle x+p\rangle^s}\frac{g(y)}{\langle y+p\rangle^s}\,dy,
\end{equation}
uniformly for $p\in\mathbb{R}^3$ and $\varep\in(0,1]$. This follows using the $TT^\ast$ argument as in Proposition \ref{prop1}, the uniform bounds in Lemma \ref{lem4} (ii), and \eqref{bil7}.
\end{proof}

\section{Proof of Theorems \ref{theorem:main} and \ref{theorem:main2}}\label{Sec:MainProof}

The two theorems follow by similar arguments from Propositions \ref{PropLinfty} and \ref{PropL2}. For concreteness, we provide all the details only for the proof of Theorem \ref{theorem:main2}. 

\begin{proof}[Proof of Theorem \ref{theorem:main2}] (i) Let $T:=A_s^{-1}R^{-2}$ for a sufficiently large constant $A_s$. We define the approximating sequence
\begin{equation}\label{cran1}
f^0(t):=f_0,\qquad f^{n+1}(t):=f_0+\int_0^t \mathcal{Q}(f^n(\tau))\,d\tau,
\end{equation}
on the interval $[0,T]$. Using Proposition \ref{PropL2} it follows easily, by induction that $f^n\in C^1([0,T]:L^2_s)$ and $\sup_{t\in[0,T]}\|f_n(t)\|_{L^2_s}\leq 2R$. Using again Proposition \ref{PropL2} it follows that the sequence $f^n$ is Cauchy in $C([0,T]:L^2_s)$, thus convergent to a function $f\in C([0,T]:L^2_s)$ that has the properties
 \begin{equation}\label{cran2}
f(0)=f_0,\qquad f(t)=f_0+\int_0^t \mathcal{Q}[f(\tau)]\,d\tau,\qquad \sup_{t\in[0,T]}\|f(t)\|_{L^2_s}\leq 2R.
\end{equation}
In particular $\partial_tf=\mathcal{Q}[f]$, thus $f\in C^1([0,T]:L^2_s)$. Uniqueness and continuity of the flow map $f_0\to f$ follow again from the contraction principle.

(ii) Clearly,  $f$ is real-valued if $f_0$ is real-valued. To prove non-negativity, we need to be slightly more careful because the simple recursive scheme \eqref{cran1} does not preserve non-negativity. 

\noindent \underline{Step 1:} We construct a different approximating sequence, based on the temporal forward Euler scheme: for any $n\in\mathbb{N}$ we set $\Delta_n=T/n$ and define the sequence $\{g^{n,m}\}_{i=0}^{n-1}$ by
\begin{equation}\label{Sec:MainProof:E4}\begin{aligned}
g^{n,0}\ & := \ f_0,\qquad g^{n,m+1}:= \ g^{n,m} + \Delta_n \mathcal{Q}[g^{n,m}].
\end{aligned}
\end{equation}
Then we define $g^n$ for $t\in [m\Delta_n, (m+1)\Delta_n]$ by the formula 
\begin{equation}\label{carc3}
\begin{split}
g^n(t)&:=g^{n,m}+(t-m\Delta_n) \mathcal{Q}[g^{n,m}]\\
&= \frac{1}{\Delta_n}\left((t-m\Delta_n) g^{n,m+1} + ((m+1)\Delta_n-t) g^{n,m}\right).
\end{split}
\end{equation}
Using Proposition \ref{PropL2} inductively and the assumption $T=A_s^{-1}R^{-2}$, it is easy to verify that
\begin{equation}\label{carc3.5}
\|g^{n,m}\|_{L^2_s}\leq 2R\qquad \text{ for any }n\geq 1\text{ and }m\in\{0,\ldots,n-1\}.
\end{equation}
In particular, using the definition \eqref{carc3},
\begin{equation}\label{carc4}
g^n\in C([0,T]:L^2_s)\text{ for any }n\geq 1\qquad\text{ and }\qquad\sup_{t\in[0,T]}\|g^n(t)\|_{L^2_s}\leq 2R.
\end{equation}

\noindent \underline{Step 2:} We show now that
\begin{equation} \label{carc8}
\lim_{n\to\infty}g^n=f\qquad\text{ in }\qquad C([0,T]:L^2_s).
\end{equation}
Let $\delta_n:=\sup_{t\in[0,T]}\|g^n(t)-f(t)\|_{L^2_s}$. Given $t\in[0,T]$ we fix $m\in\{0,1,\ldots,n-1\}$ such that $mT/n\leq t\leq (m+1)T/n$. Then we write, using \eqref{cran2}--\eqref{carc3}, 
\begin{equation}\label{cran7.5}
\begin{split}
&g^n(t)-f(t)\\
&=\{g^n(t)-g^{n,m}\}+\Big\{g^{n,m}-f_0-\int_0^{mT/n}\mathcal{Q}[f(\tau)]\,d\tau\Big\}-\int_{mT/n}^t\mathcal{Q}[f(\tau)]\,d\tau\\
&=I(t)+II(t)+III(t),
\end{split}
\end{equation}
where
\begin{equation*}
\begin{split}
I(t)&:=(t-mT/n) \mathcal{Q}[g^{n,m}],\\
II(t)&:=\sum_{j=0}^{m-1}\int_{jT/n}^{(j+1)T/n}\big\{\mathcal{Q}[g^{n,j}]-\mathcal{Q}[f(\tau)]\big\}\,d\tau,\\
III(t)&:=-\int_{mT/n}^t\mathcal{Q}[f(\tau)]\,d\tau.
\end{split}
\end{equation*}
Using Proposition \ref{PropL2} and the bounds \eqref{cran2} and \eqref{carc3.5} we estimate
\begin{equation}\label{cran7}
\|I(t)\|_{L^2_s}+\|III(t)\|_{L^2_s}\lesssim (T/n)R^3\lesssim R/n.
\end{equation}
We estimate also, for any $\tau\in[jT/n,(j+1)T/n]$,
\begin{equation*}
\begin{split}
\big\|\mathcal{Q}[g^{n,j}]-\mathcal{Q}[f(\tau)]\big\|_{L^2_s}&\lesssim \big\|\mathcal{Q}[g^{n,j}]-\mathcal{Q}[g^{n}(\tau)]\big\|_{L^2_s}\\
&+\big\|\mathcal{Q}[g^{n}(\tau)]-\mathcal{Q}[f(\tau)]\big\|_{L^2_s}\\
&\lesssim (T/n)R^5+\delta_nR^2,
\end{split}
\end{equation*}
using Proposition \ref{PropL2} and recalling the definition $\delta_n:=\sup_{t\in[0,T]}\|g^n(t)-f(t)\|_{L^2_s}$. Thus
\begin{equation}\label{cran8}
\|II(t)\|_{L^2_s}\lesssim R/n+\delta_n (TR^2).
\end{equation}
Since $TR^2\leq A_s^{-1}\ll 1$, it follows from \eqref{cran7.5}--\eqref{cran8} that $\delta_n\lesssim R/n$. The desired conclusion \eqref{carc8} follows.

\noindent \underline{Step 3:} Finally, we show that all the functions $g^n$ are non-negative. In view of the defintion \eqref{carc3}, it suffices to prove that the functions $g^{n,m}$ are non-negative for any $n\geq 1$ and $m\in\{0,\ldots,n-1\}$. We prove this by induction over $m$. The case $m=0$ follows from the hypothesis $f_0\geq 0$. Moreover, recalling the definition \eqref{4waveEquivalence},
\begin{equation*}
g^{n,m+1}\geq g^{n,m}+\Delta_n\big[\mathcal{T}_2(g^{n,m},g^{n,m},g^{n,m})+\mathcal{T}_3(g^{n,m},g^{n,m},g^{n,m})\big].
\end{equation*}
Recall that $\mathcal{T}_k(g^{n,m},g^{n,m},g^{n,m})=g^{n,m}\cdot Q_k(g^{n,m},g^{n,m})$, $k\in\{1,2\}$, see definitions \eqref{bil1} and \eqref{bil30}. Using Lemmas \ref{prop1} and \ref{prop2}, it follows that
\begin{equation*}
g^{n,m+1}\geq (1-C_sR^2T/n)g^{n,m}\geq (1-1/(2n))g^{n,m}.
\end{equation*} 
The non-negativity of the functions $g^{n,m}$ follows. This implies the non-negativity of the solution $f$, as a consequence of \eqref{carc8}. 
\end{proof}

\section{Further results}\label{Further}
Define the function space  $\mathbb{L}^r_s$ by the norm
$$\|f\|_{\mathbb{L}^r_s} \ = \ \| (1+\omega_p)^s f\|_{L^r}.$$

Notice that our theorems \ref{theorem:main} and \ref{theorem:main2} are valid for the case where the initial condition does not belong to $\mathbb{L}_1^1$. In this case, moment estimate techniques, such as those used in \cite{BriantEinav:2016:OTC,AlonsoGambaBinh} are not applicable. 

Now, if we consider the 4-wave turbulence kinetic equation \eqref{4wave} (or \eqref{4waveEquivalence}), and suppose in addition that $f_0\in \mathbb{L}_1^1$; similar to the case of the classical Boltzmann equation \cite{MR1942465}, we also have the conservation of mass, momentum and energy of  solutions to \eqref{4wave}.

Taking any \ $\varphi\in C_c(\mathbb{R}^3)$ as a test function in \eqref{4wave}, the following weak formulation holds true
\begin{equation}
\begin{aligned}\label{WeakForm}
    \int_{\mathbb{R}^3}\mathcal{Q}[f]\varphi dp   \ =   & \  \int_{\mathbb{R}^9}\delta(p+p_1-p_2-p_3)\delta(\omega+\omega_1-\omega_2-\omega_3)\times\\
& \ \times ff_1(f_2+f_3)[\varphi_2+\varphi_3-\varphi-\varphi_1]dp_1dp_2dp_3dp,
\end{aligned}
\end{equation}
in which, again, we have used the abbreviation $\varphi=\varphi(t,p)$, $\varphi_1=\varphi(t,p_1)$, $\varphi_2=\varphi(t,p_2)$, $\varphi_3=\varphi(t,p_3)$. By choosing $\varphi$ to be  $1$, $p$ or $\omega$, the right hand side of \eqref{WeakForm} vanishes. 

Since

\begin{equation*}
\label{WeakFormulation}
\begin{aligned}
\partial_t \int_{\mathbb{R}^3}f\varphi dp\ =  \ \int_{\mathbb{R}^3}\mathcal{Q}[f]\varphi dp   ,
\end{aligned}
\end{equation*}
the following conservation laws are then satisfied
\begin{equation}\label{ConservationLaws}
\partial_t \int_{\mathbb{R}^3}f dp   \ =    \ \partial_t \int_{\mathbb{R}^3}fp^idp   \ =    \ \partial_t \int_{\mathbb{R}^3}f \omega dp  \ = \ 0,
\end{equation}
with $p=(p^1,p^2,p^3)$, $i\in\{1,2,3\}$, or equivalently
\begin{equation}\begin{aligned}\label{ConservationLaws2}
\int_{\mathbb{R}^3}f(t,p)dp \ & = \ \int_{\mathbb{R}^3}f_0(p)dp,\\
\int_{\mathbb{R}^3}f(t,p)p^i dp \ & = \ \int_{\mathbb{R}^3}f_0(p)p^i dp,
\\
\int_{\mathbb{R}^3}f(t,p)\omega_{p} dp \ & = \ \int_{\mathbb{R}^3}f_0(p)\omega_{p} dp.
\end{aligned}\end{equation}

By the same argument used in (ii) of the proofs of Theorem \ref{theorem:main} and Theorem \ref{theorem:main2}, we obtain the following theorem.

\begin{theorem}\label{theorem:further} Assume that $\omega$ and the positive initial condition $f_0$ satisfy the assumptions of Theorem \ref{theorem:main} and Theorem \ref{theorem:main2}. In addition, suppose $f_0\in \mathbb{L}_1^1$. Then the same conclusion of Theorem \ref{theorem:main} and Theorem \ref{theorem:main2} holds true. Furthermore, $f\in C([0,T]: \mathbb{L}_1^1)$ and $f$ also satisfies the conservation laws \eqref{ConservationLaws2}.

\end{theorem}

\bigskip

{\bf Acknowledgment:} PG was supported by the National Science Foundation grant DMS-1301380. ADI was supported in part by NSF grant DMS-1600028
and by NSF-FRG grant DMS-1463753. MBT was supported by  NSF Grant DMS (Ki-Net) 1107291, ERC Advanced Grant DYCON.
 The authors would like to thank Sergey Nazarenko and Alan Newell for fruitful discussions on the topic.

\bibliographystyle{plain}

\bibliography{QuantumBoltzmann}

\end{document}